\documentclass[12pt, a4paper]{amsart}
\usepackage{amssymb}
\usepackage{textcomp}
\usepackage{mathrsfs}

\usepackage[all]{xy}
\CompileMatrices

\theoremstyle{plain}

\newtheorem{theorem}{Theorem}[section]
\newtheorem{lemma}[theorem]{Lemma}
\newtheorem{fact}[theorem]{Fact}
\newtheorem{corollary}[theorem]{Corollary}
\newtheorem{proposition}[theorem]{Proposition}

\theoremstyle{definition}
\newtheorem{definition}[theorem]{Definition}
\newtheorem{example}[theorem]{Example}

\newtheorem{question}[theorem]{Question}

\newtheorem{problem}[theorem]{Problem}
\newtheorem{remark}[theorem]{Remark}

\theoremstyle{remark}

\newtheorem*{claim}{Claim}
\newtheorem*{acknowledgements}{Acknowledgements}

\numberwithin{figure}{section}
\numberwithin{equation}{section}

\DeclareMathOperator{\im}{Im}
\DeclareMathOperator{\UT}{UT}
\DeclareMathOperator{\GL}{GL}
\DeclareMathOperator{\Th}{Th}

\newcommand{\G}{{\mathcal G}}
\renewcommand{\H}{{\mathcal H}}
\newcommand{\Rr}{{\mathcal R}}
\newcommand{\N}{{\mathbb N}}
\newcommand{\F}{{\mathbb F}}
\newcommand{\Z}{{\mathbb Z}}
\newcommand{\R}{{\mathbb R}}
\newcommand{\Q}{{\mathbb Q}}

\newcommand{\Lan}{{\mathcal L}}

\begin{document}

\title[Model theoretic connected components of nilpotent groups]{Model theoretic connected components \\ of finitely generated nilpotent groups}
\date{\today}

\author{Nathan Bowler \and Cong Chen \and Jakub Gismatullin}
\thanks{The second author is funded by EPSRC Doctoral Training Grant (DTG) Reference EP/P505593/1. The third author is supported by the Marie Curie Intra-European Fellowship MODGROUP no. PIEF-GA-2009-254123 and Polish Goverment MNiSW grant N N201 545938}

\address{Centre for Mathematical Sciences, University of Cambridge, Wilberforce Road, Cambridge CB3 0WA, United Kingdom.}
\email{sfwc@hotmail.com}

\address{School of Mathematics, University of Leeds, Woodhouse Lane, Leeds, LS2 9JT, UK}
\email{mmcc@leeds.ac.uk}

\address{School of Mathematics, University of Leeds, Woodhouse Lane, Leeds, LS2 9JT, UK}
\address{and}
\address{Instytut Matematyczny Uniwersytetu Wroc{\l}awskiego, pl. Grunwaldzki 2/4, 50-384 Wroc{\l}aw, Poland}
\email{gismat@math.uni.wroc.pl, www.math.uni.wroc.pl/\~{}gismat}


\begin{abstract}
We prove that for a finitely generated infinite nilpotent group $G$ with structure $(G,\cdot,\ldots)$, the connected component ${G^*}^0$ of a sufficiently saturated extension $G^*$ of $G$ exists and equals $\bigcap_{n\in\N} \left\{g^n : g\in G^*\right\}$. We construct an expansion of $\Z$ by a predicate $(\Z,+,P)$ such that the type-connected component ${\Z^*}^{00}_{\emptyset}$ is strictly smaller than ${\Z^*}^0$. We generalize this to finitely generated virtually solvable groups. As a corollary of our construction we obtain an optimality result for the van der Waerden theorem for finite partitions of groups.
\end{abstract}

\maketitle

\section{Introduction}

Thoroughout this paper we consider an infinite group $G$ and an expansion $(G,\cdot,\ldots)$ of it by some additional structure in the language $\Lan=\{\cdot,\ldots\}$. The symbol $\cdot$ is the group operation and $\ldots$ denotes some additional relations, functions and constants.
%
%Let $G$ be an arbitrary infinite group. Consider an expansion of $G$ by some structure $(G,\cdot,\ldots)$, where $\cdot$ is the usual group multiplication and $\ldots$ represents some additional structure i.e. relations, functions and constants. (in the language $\Lan=\{\cdot,\ldots\}$)
Following standard usage in model theory, $G^*$ stands for a \emph{sufficiently saturated} elementary extension of $G$: one that is $\kappa$-saturated and $\kappa$-strongly homogeneous for some sufficiently big cardinal $\kappa$.

In model theory we consider several kinds of \emph{model-theoretic connected components} of such structures $G$. Assume that $G$ is sufficiently saturated and let $A\subset G$ be a \emph{small} set of parameters, that is one of cardinality less than $\kappa$. We define (see e.g. \cite[Definition 2.1]{modcon}): 
\begin{itemize}
\item the \emph{connected component of $G$ over $A$}, denoted by $G^0_A$, as the intersection of all $A$-definable subgroups of $G$ which have finite index in $G$,
\item the \emph{type-connected component of $G$ over $A$}, denoted by $G^{00}_A$, as the smallest subgroup of bounded (that is, smaller than $\kappa$) index in $G$ that is type-definable over $A$,
\item the \emph{$\infty$-connected component of $G$}, denoted by $G^{\infty}_A$, as the smallest subgroup of bounded index in $G$ that is invariant under the automorphisms of $G$ fixing $A$ pointwise ($A$-invariant in short).
\end{itemize}
These groups are normal subgroups of $G$ and satisfy $G^{\infty}_A\subseteq G^{00}_A\subseteq G^0_A$. If for every small $A\subset G$ we have $G^0_A=G^0_{\emptyset}$, then we say that $G^0$ (without subscript) exists and we call this group the \emph{connected component of $G$}. Similarly we define the \emph{type-connected component $G^{00}$ of $G$} and the \emph{$\infty$-connected component $G^{\infty}$ of $G$}. If $G^{\infty}$ exists, then $G^{00}$ and $G^0$ also exist (see e.g. \cite[Section 5]{modcon}).

One reason for studying these components is that $G/G^{0}$, $G/G^{00}$ and $G/G^{\infty}$ with a suitable topology, called the \emph{logic topology}\footnote{see e.g. \cite[Section 2]{pill} and \cite[Proposition 3.5 1.]{gomp}, where the symbol $G_L$ is used to denote $G^{\infty}_{\emptyset}$} are compact topological groups which are invariants of the first order theory $\Th(G,\cdot,\ldots)$ of $(G,\cdot,\ldots)$. Not much is known about them for saturated elementary extensions of finitely generated groups, as the property of being finitely generated is not preserved under elementary extension.

In section 2 of this paper we define thick sets and use them to characterize the components as in \cite{modcon}.

We are interested in the class of finitely generated nilpotent groups because as we show in Corollary \ref{cor:poi} every subgroup of finite index in a group elementarily equivalent to such a group is definable. 

We can try to describe these finite index subgroups more concretely as follows. Every finite index subgroup of a finitely generated group is itself finitely generated \cite[14.3.2]{km}. The groups $\UT_n(\Z)$ of unitriangular matrices over $\Z$ are examples of torsion-free finitely generated nilpotent groups. In fact, every torsion-free finitely generated nilpotent group can be homomorphically embedded into $\UT_n(\Z)$, for suitable natural $n$ \cite[17.2.5]{km}. In general, by \cite[17.2.2]{km}, an arbitrary finitely generated nilpotent group has a finite index subgroup which is torsion-free (also finitely generated and nilpotent).

%In Section 3 we investigate when ${G^*}^0$ exists and is divisible, giving an explicit description for finitely generated nilpotent groups, and prove that certain properties cannot be generalized to virtually solvable groups.

A group is virtually nilpotent when it has a nilpotent subgroup of finite index; similarly it is virtually solvable when it has a solvable subgroup of finite index.  Putting all  of this together, we see that $G$ is a finitely generated virtually nilpotent group if and only if it has a finite index subgroup which can be homomorphically embedded into one of the groups $\UT_n(\Z)$.

We describe when ${G^*}^0$ exists and prove in Theorem \ref{thm:divisible} that in an infinite finitely generated virtually nilpotent group ${G^*}^0$ is divisible and of the form \[\bigcap_{n\in\N} \left\{g^n : g\in G^*\right\}.\] We show in Corollary \ref{cor:cong} that ${G^*}^0$ cannot be divisible for a virtually solvable group $G$ unless it is virtually nilpotent, using results from \cite{hrush}.

Consider the additive group of integers $(\Z,+)$. Since the theory of $(\Z,+)$ is stable, it is well known that in a saturated extension $\Z^*$ of $(\Z,+)$, the groups ${\Z^*}^0$, ${\Z^*}^{00}$ and ${\Z^*}^{\infty}$ exist and coincide with $\bigcap_{n\in\N} n{\Z^*}$. Furthermore, the same is true for any definable group $G$ in $(\Z,+)$: ${G^*}^0={G^*}^{00}={G^*}^{\infty}$. In particular, one can apply this to $G=(\Z^n,+)$. We show that exactly the same is true for Presburger arithmetic in \ref{cor:pres}.

We identify groups admitting homomorphisms with dense image into $\R/\Z$ in section 4. These include abelian groups of unbounded exponent. We show in Proposition \ref{prop:dense} that these homomorphisms allow us to construct expansions of groups such that ${G^*}^{00}_{\emptyset} \lneqq {G^*}^0_{\emptyset}$. We do this by adding a new predicate which gives new type-definable subgroups of finite index. We then construct such homomorphisms for $(\Z, +)$ and more general groups. In an extension of this for finitely generated virtually solvable groups we find finite index subgroups admitting those maps (\ref{cor:ab}). In particular this shows that if we expand $(\Z,+)$ by a predicate $(\Z,+,P)$ the type-connected components are no longer easily characterised. Such a predicate can be defined in First Order Arithmetic, giving a similar result (\ref{cor:peano}).

% perhaps state a couple of results

%In Section 4 we use the description from Section 2 to illustrate a more general result distinguishing $G^{00}$ from $G^0$ in certain abelian and virtually solvable groups.

%In Section 5 we present some results in the spirit of the van der Waerden's theorem about finite partitions of groups. We use thick subsets which are far from containing subgroups.

In section 5 we apply some techniques developed earlier in the paper to show some results using no additional model theory. These provide some limits to claims in the style of the van der Waerden theorem about explicit partitions that do not contain certain subgroups, first for $\Z$ and then in more general cases. We do, however, obtain a positive result of this kind in groups whose cardinality is greater than that of the continuum.

\section{Basic results}

%We assume that $(G,\cdot,\ldots)$ is a group with some additional first order structure in the language $\Lan$. Throughout the rest of the paper we will denote by $G^*$ a sufficiently saturated elementary extension of $G$ and by $A\subset G^*$ a small set of parameters.

For a subset $X$ of a group $G$ we write $X^n$ for $\underbrace{X\cdot \ldots \cdot X}_{n\text{ times}}$.

In what follows we need another description of ${G^*}^{\infty}_A$ (Fact \ref{fact:known}$(2)$), which comes from \cite[Section 3]{modcon}: this subgroup can be described in terms of definable thick subsets.

\begin{definition}{\cite[Definition 3.1]{modcon}} Suppose $P$ is a subset of an arbitrary group $G$ and $n\in\N$.
\begin{enumerate}
\item $P$ is called \emph{$n$-thick} if it is symmetric (that is $P = P^{-1}$) and for every $n$-sequence $g_1,\ldots,g_n$ of elements of $G$, there are $1\leq i<j \leq n$ such that $g_i^{-1}g_j \in P$ (we do not require $g_0,\ldots,g_{n}$ to be pairwise distinct). $P$ is \emph{thick} if it is $n$-thick for some natural $n$.
\item $P$ is \emph{right [resp. left] $n$-generic} if at most $n$ right [resp. left] translates of $P$ by elements of $G$ cover the whole group $G$. $P$ is \emph{generic} if it is right $n$-generic for some natural $n$.
\end{enumerate}
\end{definition}

We recall some basic properties of thick sets (see \cite[Lemma 3.2]{modcon}).

\begin{lemma} \label{lem:thick}
\begin{enumerate}
\item The intersection of any two thick subsets is thick. More precisely if $P$ is $n$-thick and $Q$ is $m$-thick, then $P\cap Q$ is $R(n,m)$-thick, where $R(n,m)$ is a Ramsey number.
\item If $H$ is a subgroup of $G$ and $P\subseteq G$ is thick, then $P\cap H$ is thick in $H$.
\item The preimage of any $n$-thick set under any homomorphism is again $n$-thick. The image of any $n$-thick set under any surjective homomorphism is again $n$-thick.
\item For $n\geq 2$, if $P$ is $n$-thick, then $P$ is right and left $(n-1)$-generic.
\item If $P$ is right [resp. left] $n$-generic, then $P^{-1}P$ [resp. $PP^{-1}$] is $(n+1)$-thick.
\end{enumerate}
\end{lemma}
\begin{proof} The proof of $(1)-(3)$ is standard. $(5)$ follows by a pigeonhole principle. The proof of $(4)$ is in \cite[Lemma 3.2(4)]{modcon}.
\end{proof}

If a subgroup $H$ of $G^*$ is type-definable over $A$, then one can find a family of $A$-definable subsets $\{P_i\}_{i\in I}$ of $G^*$ such that $H = \bigcap_{i\in I}P_i$ and
\begin{enumerate}
\item $P_i = P_i^{-1}$, for $i\in I$,
\item for every $i\in I$, there is $j\in I$ such that $P_j\cdot P_j\subseteq P_i$.
\end{enumerate}
Conversely, if a family of $A$-definable subsets $\{P_i\}_{i\in I}$ satisfies $(1)$ and $(2)$, then $\bigcap_{i\in I}P_i$ is a type-definable subgroup of $G^*$.

We shall use the following result from \cite{modcon}.

\begin{lemma}{\cite[Lemma 3.6]{modcon}} \label{lem:gen}
If $P$ is a right $m$-generic subset of a group $G$, $1 \in P$ and $P = P^{-1}$, then $P^{3m-2}$ is a subgroup of $G$ of index at most $m$.
\end{lemma}

\begin{fact} \label{fact:known}
\begin{enumerate}
\item Suppose $H = \bigcap_{i\in I}P_i$ is a subgroup of $G^*$ type-definable over $A$, where $P_i=P_i^{-1}$ for $i\in I$. Then the index $[G^*:H]$ is bounded if and only if for each $i\in I$, $P_i$ is thick.
\item The component ${G^*}^{\infty}_A$ is the subgroup generated by the intersection of all $A$-definable thick subsets of ${G^*}$.%, that is ${G^*}^{\infty}_{A} = \left\langle\bigcap\left\{P : P\subseteq G^* \text{ is $A$-definable and thick}\right\}\right\rangle$.
\end{enumerate}
\end{fact}
\begin{proof} $(1)$ If $P_i$ is not thick, then by compactness, there is an infinite sequence $(g_i)_{i<\kappa}\subset G^*$ of length $\kappa$ (where $\kappa$ is the saturation of $G^*$) with $g_i^{-1}g_j\not\in P_i$. Hence $[G^*:H]\geq\kappa$ is not bounded.

Suppose now that for each $i\in I$, $P_i$ is thick. Let $M$ be an elementary submodel of $G^*$ containing $A$ and of cardinality at most $|A|+|\Lan|$ (such an $M$ exists by the downward L\"owenheim-–Skolem Theorem). It is enough to prove that if $a,b\in G^*$ have the same type over $M$, that is $a\underset{M}{\equiv} b$, then $a^{-1}b\in H$, because then $[G^*:H]\leq |S_1(M)|\leq 2^{|A|+|\Lan|}$, and $H$ would have bounded index.

Suppose that $a\underset{M}{\equiv} b$. Then by \cite[1.11(ii), 1.12]{pillay}, there exist $c\in G^*$ and sequences $(a_n)_{n\geq 2}$, $(b_n)_{n\geq 2}$ from $G^*$ such that $a,c,a_2,a_3,\ldots$ and $c,b,b_2,b_3,\ldots$ are indiscernible sequences over $A$. Fix $k\in I$. Since $P_k$ is thick, there are $i<j$ such that $a_i^{-1}a_j\in P_k$, and by indiscernibility, $a^{-1}c\in P_k$. Similarly $c^{-1}b\in P_k$. Hence $a^{-1}b\in\bigcap_{k\in I} P_k^2=H^2=H$.

$(2)$ is \cite[Lemmas 2.2(2), 3.3]{modcon}. 
\end{proof}

\section{$G^0$ in finitely generated nilpotent groups and divisibility}

%Let $G$ be a finitely generated group with nilpotent subgroup of finite index (virtually nilpotent group). Martinez \cite{mar94} proved that the set of $n$th powers $\left\{ g^n : g\in G \right\}$ generates the subgroup of $G$ in finitely many steps. From a model theoretical point of view, this gives both the existence and a description of $G^0$. Namely, (see Proposition \ref{prop:mar}$(2)$ below) for an arbitrary first order expansion $(G,\cdot,\ldots)$ of $G$, component ${G^*}^{0}$ exists in a saturated extension $G^*$ and ${G^*}^{0} =\bigcap_{n\in\N} \left\langle\left\{ g^n : g\in G^* \right\}\right\rangle$. The next step is to find a similar description of ${G^*}^{00}$ and ${G^*}^{\infty}$.

It is a well known fact that for an arbitrary expansion $G = (\Z,+,\ldots)$ of the additive group of the integers $\Z$, the component ${G^*}^0$ exists and equals $\bigcap_{n\in\N}n\Z^*$. The next proposition gives a sufficient condition for a similar description of ${G^*}^0$ to hold for an arbitrary infinite group.

\begin{proposition} \label{prop:0} 
Suppose that for every natural $n$ the set of $n$-th powers $\left\{g^n : g\in G\right\}$ generates a subgroup of $G$ of finite index in finitely many steps. Then
\begin{enumerate}
\item ${G^*}^0$ exists and ${G^*}^0 = \bigcap_{n\in\N} \left\langle\left\{g^n : g\in G^*\right\}\right\rangle$,
\item every subgroup of finite index in a group elementarily equivalent to $G$ is definable in the language of groups.
\end{enumerate}
\end{proposition}
\begin{proof} Since $\left\langle\{g^n : g\in G^*\}\right\rangle$ is a $\emptyset$-definable subgroup of $G^*$ of finite index, it contains ${G^*}^0_A$ for every small $A$. This gives $\subseteq$ in $(1)$. Let $H$ be an arbitrary subgroup of finite index in a group $G'$ elementary equivalent to $G$. It is enough to find $n$ such that $\{g^n : g\in G'\}\subseteq H$, because then $H$ is a finite union of cosets of $\left\langle\left\{g^n : g\in G^*\right\}\right\rangle$. Let $N$ be the intersection of all conjugates of $H$ in $G'$ (note that $H$ has only finitely many conjugates in $G'$). Then, as $N$ is a normal subgroup of $G'$, for $n=[G':N]$ we have $\left\langle\{g^n : g\in G^*\}\right\rangle \subseteq N \subseteq H$.
\end{proof}

If $G$ has no subgroups of finite index, then ${G^*}^0=G^*$ where the saturation may be with respect to some further structure in addition to the group structure. Divisible groups are examples of groups satisfying the condition from Proposition \ref{prop:0}. Furthermore, we have the following remark.

\begin{remark}
If $G$ has no proper subgroups of finite index and the commutator subgroup $[G,G]$ is finite of cardinality $N$, then $G = \{g^n : g\in G\}^{3N-2}$ for each natural $n$.
\end{remark}
\begin{proof}
Consider the abelianization $G^{\rm ab} = G/[G,G]$. It has no proper subgroups of finite index, so must be divisible (because $G^{\rm ab}/pG^{\rm ab}$ for prime $p$ is a vector space over the finite field $\F_p$). Therefore, for each natural $n$, $G = \{g^n : g\in G\}\cdot [G,G]$. By Lemma \ref{lem:gen}, $G = \{g^n : g\in G\}^{3N-2}$.
\end{proof}

%A group is \emph{virtually solvable} [resp. \emph{virtually nilpotent}] if it has a subgroup of finite index that is solvable [resp. nilpotent]. 

The condition of Proposition \ref{prop:0} is satisfied by every finitely generated virtually nilpotent group. This can be derived from a result of Martinez \cite{mar94}, which essentially uses the positive solution of the restricted Burnside problem.

\begin{proposition} \label{prop:mar} Let $G$ be a finitely generated infinite group.
\begin{enumerate}
\item[(1)] If $G$ is virtually solvable, then for each natural $n$ the group $\left\langle\{g^n : g\in G\}\right\rangle$ has finite index in $G$.
\item[(2)] If $G$ is virtually nilpotent, then $\{g^n : g\in G\}$ generates a subgroup in finitely many steps.
\end{enumerate}
\end{proposition}
\begin{proof}
$(1)$ This fact is well known \cite[Lemma 1]{mal}. We may assume that $G$ is solvable. By consideration of the quotient \[G/\left\langle\{g^n : g\in G\}\right\rangle,\] it is enough to prove that every finitely generated solvable group of finite exponent is finite. Let $G$ be such a group. The group $G/[G,G]$ is abelian, finitely generated and of finite exponent, so $G/[G,G]$ is finite. The group $[G,G]$ is finitely generated (see \cite[14.3.2]{km}). The degree of solvability of $[G,G]$ is less than that of $G$, hence by the induction hypothesis, $[G,G]$ is finite, so $G$ is finite too. 

$(2)$ Let $H<G$ be a nilpotent subgroup of finite index. By \cite[Lemmas 1,2 and 3]{mar94}, which uses the restricted Burnside problem, the set $\{h^n : h\in H\}$ generates a group in finitely many steps, say in $N$ steps. By $(1)$ the group $\left\langle\{g^n : g\in H\}\right\rangle$ has finite index in $\left\langle\{g^n : g\in G\}\right\rangle$. Take $g_1,\dots,g_d \in \left\langle\{g^n : g\in G\}\right\rangle$ as representatives of the cosets of $\left\langle\{h^n : h\in H\}\right\rangle$. If each $g_i$ is a product of $M$ $n$-th powers, then each element of $\left\langle\{g^n : g\in G\}\right\rangle$ is a product of $N+M$ $n$-th powers.
\end{proof}

Using \ref{prop:0}$(2)$ and \ref{prop:mar} we obtain the following corollary.

\begin{corollary} \label{cor:poi}
Every subgroup of finite index in a group elementarily equivalent to a finitely generated virtually nilpotent group is definable in the language of groups.
\end{corollary}

Note that such a finite index subgroup might not be $\emptyset$-definable. For example $\Z\times 2\Z$ is not $\emptyset$-definable in $(\Z^2,+)$, as $\Z\times 2\Z$ and $2\Z\times \Z$ are isomorphic through an automorphism of $\Z^2$.

As a corollary of Propositions \ref{prop:0} and \ref{prop:mar} we have that ${G^*}^0 = \bigcap_{n\in\N} \left\langle\left\{g^n : g\in G^*\right\}\right\rangle$ for a finitely generated virtually nilpotent group $G$. In Theorem \ref{thm:divisible}, using a result of Mal'cev, we shall prove that in this case ${G^*}^0 = \bigcap_{n\in\N} \left\{g^n : g\in G^*\right\}$.

The component ${\Z^*}^0$ is clearly a divisible group. We would like to know whether the components of nilpotent groups are also divisible. The following remark gives a necessary condition for this to happen.

\begin{remark} \label{rem:gen}
Assume that for some $n\in\N_{\geq 2}$ and some small parameter set $A$, at least one of the components ${G^*}^0_A$, ${G^*}^{00}_A$ or ${G^*}^{\infty}_A$ is $n$-divisible. Then the set of $n$-th powers $\{g^n : g\in G\}$ is thick.
\end{remark}
\begin{proof}
Suppose for example that ${G^*}^{\infty}_A$ is $n$-divisible. Then ${G^*}^{\infty}_A\subseteq \{g^n : g\in G^*\}$. Note that the set $\{g^n : g\in G^*\}$ is $\emptyset$-definable in $G$. By Fact \ref{fact:known}$(2)$ \[\bigcap \left\{P : P\subseteq G^* \text{ is $A$-definable and thick}\right\}\subseteq \{g^n : g\in G^*\},\] so by a compactness argument $P\subseteq \{g^n : g\in G^*\}$, for some $A$-definable thick $P$. Hence $\{g^n : g\in G^*\}$ is thick, so also $\{g^n : g\in G\}$ is thick.
\end{proof}

\begin{remark} \label{rem:powers}
Suppose that for every $n\in\N$ the set $\{g^n : g\in G\}$ is thick. Then ${G^*}^0$ exists and ${G^*}^0 = \bigcap_{n\in\N} \left\langle\left\{g^n : g\in G^*\right\}\right\rangle$.
\end{remark}
\begin{proof}
The conclusion follows by Proposition \ref{prop:0} and Lemma \ref{lem:gen}.
\end{proof}

In addition to the necessary conditions given by Remark \ref{rem:gen}, we can also provide some sufficient conditions for divisibility of connected components, at least in the nilpotent case. We will need the following fundamental result of Mal'cev.

The \emph{lower central series} of a group $G$ is defined thus: $G_1 = G$, $G_{n+1}=[G,G_n]$. If $G_{m+1}$ is trivial we say that $G$ is \emph{nilpotent of class $m$}.

\begin{fact}{\cite[Lemma 2]{mal}} \label{fact:mal}
If $G$ is a nilpotent group of class $m\geq 1$, then for any $n\in\N$ the group $\left\langle\left\{g^{n^m} : g\in G\right\}\right\rangle$ is $n$-divisible.
\end{fact}

\begin{theorem} \label{thm:divisible}
Suppose that $G$ is an infinite finitely generated virtually nilpotent group. Then ${G^*}^0$ is divisible and ${G^*}^0 = \bigcap_{n\in\N} \left\{g^n : g\in G^*\right\}$.
\end{theorem}
%The part $(2)$ is not an immediate consequence of $(1)$ as subgroups of divisible groups may not be divisible, for example $\Z$ and $\Q$.

\begin{proof} Let $H<G$ be a nilpotent subgroup of finite index. By Propositions \ref{prop:0} and \ref{prop:mar}, we may assume that for some $n\in\N$ we have $H=\left\langle \left\{g^{n}:g\in G\right\}\right\rangle$, so $H$ is $\emptyset$-definable. Therefore we may assume that $G$ is a nilpotent group of class $m\geq 1$.

Let $r\in{G^*}^0$, $n\in\N$ and define $P_k = \left\langle\left\{g^{k} : g\in G\right\}\right\rangle$, for $k\in\N$ (the interpretation of $P_k$ in $G^*$ we also denote by $P_k$). Consider the following type \[p(x)=\left\{r=x^n,\  x\in {G^*}^0=\bigcap_{k\in\N}P_k \right\}.\] By Fact \ref{fact:mal} we have $\{g^n : g\in P_k\} \supseteq \{g^n : g\in P_{(kn)^m}\} = P_{(kn)^m}$. Hence $p$ is consistent, so ${G^*}^0$ is $n$-divisible. Thus, ${G^*}^0 \subseteq \bigcap_{n\in\N} \left\{g^n : g\in G^*\right\} \subseteq \bigcap_{n\in\N} P_n = {G^*}^0$.
\end{proof}

\begin{proposition}\label{prop:divisible}
Suppose $G$ is an infinite abelian group and fix a parameter set $A\subset G^*$. Then ${G^*}^{00}_A$ is divisible if and only if for every $n\in\N$, $[G:nG]$ is finite.
\end{proposition}
\begin{proof} By Fact \ref{fact:known}$(1)$, ${G^*}^{00}_A = \bigcap_{i\in I}P_i$, for some collection $\{P_i : i\in I\}$ of $A$-definable thick sets. Fix $n\in\N$.

$(\Leftarrow)$  Then by compactness $n\cdot {G^*}^{00}_A = \bigcap_{i\in I}n\cdot P_i$. Since $n\cdot P_i$ is thick in $n\cdot G^*$ and the latter has finite index in $G^*$, the set $n\cdot P_i$ is $A$-definable and thick in $G^*$. Therefore, $n\cdot {G^*}^{00}_A$ has bounded index in $G^*$, so $n\cdot {G^*}^{00}_A = {G^*}^{00}_A$.

$(\Rightarrow)$ As ${G^*}^{00}_A = n\cdot {G^*}^{00}_A$, we have by compactness that $n\cdot P_0$ is thick in $G^*$. We know that $n\cdot P_0$ is thick in $n\cdot {G^*}^{00}_A$, so $[G:nG]$ must be finite.
\end{proof}

It is a corollary from the above that if $G$ is finitely generated, then  ${G^*}^{00}_A$ is divisible. In this restricted case one can also prove that ${G^*}^{\infty}_A$ is divisible. 

%Note that while any element of ${G^*}^{0}$ (and thus ${G^*}^{00}_A$ or ${G^*}^{\infty}_A$) is divisible in ${G^*}^{0}$ it is not immediate that ${G^*}^{00}_A$ or ${G^*}^\infty_A$ contain those divisors.

%The subgroup ${G^*}^{\infty}_A$ is generated by $X_{\Theta_A}=\bigcap \{P : P\subseteq G^* \text{ is $A$-definable and thick}\}$ (Fact \ref{fact:known}$(2)$). As $G$ is abelian, the set $n\cdot P$ is $A$-definable and thick if $P\subseteq G^*$ is such. Hence $X_{\Theta_A}$ is divisible, so as $G$ is abelian, ${G^*}^{\infty}_A=\langle X_{\Theta_A} \rangle$ is also divisible.

When $G$ is finitely generated these proofs follow from the following consequence of the classification of finitely generated abelian groups: 
\begin{quote}
Let $G$ be abelian and finitely generated. If $P\subseteq G$ is thick, then for all $n\in\N$ the set $\{g^n : g\in P\}$ is also thick.
\end{quote}
We do not know whether a similar fact can be proved for nilpotent groups.

\begin{problem}
Let $G$ be a finitely generated virtually nilpotent group. Are the components ${G^*}^{00}_A$ and ${G^*}^{\infty}_A$ divisible?
\end{problem}

% try by induction
% we can prove that ${G^*}^{00}_A$ is divisible if and only if the center $Z({G^*}^{00}_A)$ is divisible

In general ${G^*}^0$ need not be divisible. In particular, the previous proposition cannot be generalized to virtually solvable groups. To prove this, we use a result from \cite{hrush}, which is a partial converse to Mal'cev's theorem (Fact \ref{fact:mal}). 

\begin{theorem}{\cite[Theorems A and B]{hrush}} \label{thm:hrush}
Suppose that $G$ is a finitely generated group.
\begin{enumerate}
\item If $G$ is virtually solvable, and for some finite $J\subset\N_{\geq 2}$, the set $\{g^n : g\in G, n\in J\}$ contains a finite index subgroup, then $G$ is virtually nilpotent.
\item Assume that $G$ is a linear group, i.e. $G$ is a subgroup of $\GL_n(R)$ for some finitely generated domain $R$. If, for some finite $J\subset\N_{\geq 2}$, the set $\{g^n : g\in G, n\in J\}$ is generic, $G$ is virtually solvable.
\end{enumerate}
\end{theorem}

The following corollary follows now by Remark \ref{rem:gen} and Lemma \ref{lem:thick}$(4)$.

\begin{corollary} \label{cor:cong}
Suppose that $G$ is a finitely generated infinite group.
\begin{enumerate}
\item If $G$ is a linear group and for some $n\in\N_{\geq 2}$, some of the components ${G^*}^0_A$, ${G^*}^{00}_A$ or ${G^*}^{\infty}_A$ are $n$-divisible, then $G$ is virtually solvable.
\item If $G$ is a virtually solvable group and ${G^*}^0_A$ is $n$-divisible for some $n\in\N_{\geq 2}$, then $G$ is virtually nilpotent, so it is a finite extension of a subgroup of $\UT_m(\Z)$, for some $m\in\N$.
\item If $G$ is an infinite finitely generated virtually solvable group, but not virtually nilpotent, then ${G^*}^0_A$ is not divisible. In fact ${G^*}^0_A$ is not $n$-divisible, for any $n\in\N_{\geq 2}$.
\end{enumerate}
\end{corollary}

% Baumslag-Solitar groups, Moldovanskii results

We finish this section with an example of the additive group of integers with order. The theory of $(\Z, +, -, <, 0, 1)$ is \emph{Presburger arithmetic} and admits quantifier elimination up to introduction of predicates for sets of the form $n\Z$ for standard $n\in\N$ \cite[Section 3.1]{marker}. This is sufficient for the following conclusions.

\begin{proposition}
Consider the group $G = (\Z^m,+)$ with additional structure induced from $(\Z, +, -, <, 0, 1)$. Any definable (with parameters) and thick subset $P$ of $G^*$ satisfies $P^{2m} \supseteq n\cdot{\Z^*}^m$, for some $n\in\N$.
\end{proposition}

\begin{proof} First consider the case $n = 1$. It follows by quantifier elimination that every definable subset $P$ of ${\Z^*}$ can be expressed as a finite union of the form \[\bigcup_i (a_i + b_i {\Z^*}) \cap (c_i, d_i),\] where $c_i, d_i$ may be $\pm\infty$ and any $b_i$ is standard. If each interval in the union is bounded then $P$ is a bounded set and cannot be thick. Hence if $P$ is thick, it contains some set of the form $(a + b {\Z^*}) \cap (c, d)$ where $(c,d)$ is an unbounded interval, and $P^2$ contains $b {\Z^*}$.

For $n > 1$ and thick $P \subseteq {\Z^*}^m$, consider the restrictions to the axes $P_i = P\cap{\Z^*}_i$, where ${\Z^*}_i$ is in the $i$th axis of ${\Z^*}^m$. These ${\Z^*}_i$ are subgroups of ${\Z^*}^m$, so by Lemma \ref{lem:thick}$(2)$ the restrictions $P_i$ are thick definable one-dimensional subsets, and thus their difference contains one-dimensional lattices in the appropriate directions. Thus $P^{2n}$ contains a $n$-dimensional lattice of the form required.
\end{proof}

If a thick set $P$ is a subset of the standard integers it will contain a copy of $n\Z$, but this may not be true if we do not work in the standard integers. Indeed, for any nonstandard integer $a$ we have that $P = (-\infty, -3a] \cup (-2a, 2a) \cup [3a, \infty)$ is $3$-thick but does not contain any $n\Z^*$.

Fact \ref{fact:known}$(2)$ implies the following.

\begin{corollary} \label{cor:pres}
For a group $G = (\Z^m,+)$ with structure as in the previous proposition, components ${G^*}^0$, ${G^*}^{00}$, and ${G^*}^{\infty}$ exist and ${G^*}^{\infty}={G^*}^{00}={G^*}^0=\bigcap_{n \in \N}n{\Z^*}^m$.
\end{corollary}

\section{Type-definable subgroups of bounded index}

Consider the additive group of the integers $(\Z,+)$. By Proposition \ref{prop:mar}$(2)$, the component ${\Z^*}^0$ is $\bigcap_{n\in\N} n{\Z^*}$ no matter what the structure on $\Z$ is. This group would also be the natural candidate for ${\Z^*}^{00}_{\emptyset}$ and ${\Z^*}^{\infty}_{\emptyset}$. However, we will find an expansion $\G = (\Z,+,\ldots)$ of $\Z$ such that ${\G^*}_{\emptyset}^{00}$ is strictly smaller that ${\G^*}^0={\Z^*}^0$.

In fact, since our results can be applied to a wider class of groups, we will start by proving a relatively general theorem.

In this section we identify $\R/\Z$ with the interval $\left[-\frac{1}{2},\frac{1}{2}\right)$.

\begin{lemma} \label{lem:gint}
Suppose $G$ is a group and $g\colon G \to \R/\Z$ is a homomorphism with dense image $\im(g)\subseteq \R/\Z$. Define $X(t)=g^{-1}(-t,t)$, for $t\in \left(0,\frac{1}{2}\right]$.
\begin{enumerate}
\item The set $X(t)$ is $\left(\left\lfloor\frac{1}{2t}\right\rfloor+1\right)$-thick in $G$.
\item $X(t_1)\cdot X(t_2) = X(t_1+t_2)$, for $t_1,t_2\in\R_{>0}$ and $t_1+t_2\leq\frac{1}{2}$.
\item For $0<t\leq\frac{1}{3}$, the kernel $\ker(g)$ is the maximal subgroup of $G$ contained in the set $X(t)$.
\item $X\left(\frac{t}{m}\right) = X(t) \cap \left\{x\in G: m\cdot x\in X(t)\right\}$ for $m\in\N_{\geq 1}$ and $0\leq t\leq\frac{1}{m+1}$.
\item For $0<t\leq \frac{1}{3}$ and every $q\in\Q\cap \left(0,\frac{1}{2t}\right]$, the set $X(qt)$ is $\emptyset$-definable in the structure $(G, \cdot, X(t))$, where we regard $X(t)$ as a predicate.
\end{enumerate}
\end{lemma}
\begin{proof} $(1)$ follows by Lemma \ref{lem:thick}$(3)$ and the fact that $(-t,t)$ is a $\left(\left\lfloor\frac{1}{2t}\right\rfloor+1\right)$-thick subset of $\R/\Z$.

The inclusion $\subseteq$ in $(2)$ is obvious. We prove $\supseteq$. Suppose $x\in X(t_1+t_2)$. We may assume that $t_2\leq t_1$ and $t_1 < g(x) < t_1 + t_2$. Thus the set $I = (-t_1,t_1)\cap(g(x)-t_2,g(x)+t_2)$ is nonempty. Since $\im(g)$ is dense in $\R/\Z$ and $I$ is open, there exists $y\in G$ with $g(y)\in I$. That means that $y\in X(t_1)$ and $g(y^{-1}x)\in(-t_2,t_2)$, so $y^{-1}x\in X(t_2)$.

To prove $(3)$, note that, since  $(-t,t)$ does not contain any nontrivial subgroup of $\R/\Z\cong \left[-\frac{1}{2},\frac{1}{2}\right)$, for $t\leq\frac{1}{3}$, $\ker(g)$ is the maximal subgroup of $G$ contained in $X(t)$. %Since element $m\in\N$ has infinite order in $\R/\alpha\Z$, one can find natural $i<j$ such that images of $im$ and $jm$ in $\R/\alpha\Z$ are at small distance $\epsilon < \alpha-2t$ to each other, with respect to the metric inherited from $\left[-\frac{\alpha}{2},\frac{\alpha}{2}\right)$. Therefore, there exists a natural number $l$ such that element $l(j-i)m$ of $\R/\alpha\Z$ is not in $\left(-t,t\right)$, so $m\Z$ is not contained in $X_{\alpha}(t)$.

For $(4)$, if $g(n)$ and $mg(n)$ are in $(-t,t)$ on $\R/\Z$ and $t \leq \frac{1}{m+1}$, then $g(n)$ is in $\left(-\frac{t}{m},\frac{t}{m}\right)$.

$(5)$ follows by $(4)$ and $(2)$. Indeed, first note that by $(4)$, the set $X\left(\frac{t}{2}\right)$ is definable, so again by $(4)$, $X\left(\frac{t}{6}\right)$ is definable. In general, $X\left(\frac{t}{m!}\right)$ is definable for every natural $m$ and by $(2)$, $X\left(\frac{t}{m}\right)$ is definable.
\end{proof}

\begin{example} \label{ex:alpha}
For an irrational real number $\alpha$ consider $g_{\alpha}\colon \Z \to \R/\Z$ defined as $g_{\alpha}(n)=n\cdot\alpha \mod 1$. Clearly $\im(g_{\alpha})$ is dense in $\R/\Z$. The previous lemma can be applied to the sets $X_{\alpha}(t) =g_{\alpha}^{-1}\left(-t,t\right)$. In fact, in this case we can improve the bound in $(3)$ of Lemma \ref{lem:gint} to $t < \frac{1}{2}$ as the subgroup generated by an irrational is dense. 
%Intuitively, for $t=\frac{1}{2}$ and big enough $\alpha$, the set $X_{\alpha}(t)$ consists of all the rounded multiples $[\alpha n]$, $n\in\Z$ of $\alpha$.
\end{example}

\begin{proposition} \label{prop:dense}
Suppose $G$ is an infinite group and $g\colon G \to \R/\Z$ is a homomorphism with dense image $\im(g)\subseteq \R/\Z$. Set $X(t)=g^{-1}(-t,t)$, for $t\in \left(0,\frac{1}{2}\right]$. Consider a structure $\G=(G,\cdot,X(t))$, for some $t\leq\frac{1}{3}$, and a $\emptyset$-type-definable subgroup \[\widetilde{\G} = \bigcap_{n\in\N} X \left(\frac{t}{n}\right)\] of $\G^*$. Then $\widetilde{\G}$ has bounded index in $\G^*$ and for every small set $A$, $\widetilde{\G}\cap{\G^*}^{0}_A \lneqq {\G^*}^{0}_A$. In particular ${\G^*}^{00}_A \lneqq {\G^*}^{0}_A$.
\end{proposition}

We denote by $X(t)$ both the subset of $\G$ and its interpretation in $\G^*$.

\begin{proof}
By $(5)$ of Lemma \ref{lem:gint}, each  $X\left(\frac{t}{n}\right)$ is $\emptyset$-definable in $\G$. Note that the set $\widetilde{\G}$ is a subgroup of $\G^*$. Indeed, by Lemma \ref{lem:gint}$(2)$, $\widetilde{\G}$ is closed under the group operation. The sets $X\left(\frac{t}{n}\right)$ are thick in $\G$ (Lemma \ref{lem:gint}$(1)$), so the group $\widetilde{\G}$ is type-definable over $\emptyset$ of bounded index in $\G^*$ (Lemma \ref{fact:known}$(1)$), hence ${\G^*}^{00}_A \leq \widetilde{\G}\cap {\G^*}^{0}_A$.

Furthermore, $\widetilde{\G}\cap{\G^*}^{0}_A \lneqq {\G^*}^{0}_A$, because otherwise $X(t) \supseteq \widetilde{\G} \supseteq {\G^*}^0_A$, so by a compactness argument $X(t)$ in $\G^*$ contains some $A$-definable finite index subgroup (actually defined over some finite subset of $A$). Then $X(t)$ in $\G$ also contains some finite index subgroup (because this fact can be written as a formula without parameters). However, by Lemma \ref{lem:gint}$(3)$, $\ker(g)$ is the maximal subgroup contained in $X(t)$, and it has infinite index in $\G$.
\end{proof}

\begin{corollary} \label{cor:intt}
Suppose that $\alpha\in\R$ is irrational and $0 < t  < \frac{1}{2}$. Let $\G^*$ be a sufficiently saturated extension of the structure $\G = (\Z,+,X_{\alpha}(t))$, where $X_{\alpha}(t)$ is from Example \ref{ex:alpha}. Then ${\G^*}^{00}_{\emptyset}$ is contained in a $\emptyset$-type definable subgroup $\widetilde{\G} = {\G^*}^0 \cap \bigcap_{n\in\N}X_{\alpha}\left(\frac{t}{n}\right)$ of $\G^*$ of bounded index, which is strictly smaller than ${\G^*}^0 = \bigcap_{n\in\N} n{\Z^*}$.
\end{corollary}

\begin{corollary} \label{cor:peano}
Consider the additive group $G=(\Z,+)$ of First Order Arithmetic $(\Z,+,\cdot)$. Then ${G^*}^{00}_{\emptyset} \lneqq {G^*}^{0}=\bigcap_{n\in\N}n\Z^*$.
\end{corollary}
\begin{proof}
First, note that as every positive integer is a sum of four squares, the order $<$ is definable in $(\Z,+,\cdot)$. It is immediate that $\exists m \left|m-n\cdot\sqrt{2}\right| < \frac{1}{3}$ is equivalent to $\exists m (3m-1)^2 < 2\cdot 9n^2 < (3m+1)^2$. Therefore the set \[X_{\sqrt{2}}\left(\frac{1}{3}\right) = \left\{n\in\Z:\exists m\in\Z,\ \ \left|m-n\cdot\sqrt{2}\right| < \frac{1}{3}\right\}\] is definable in $(\Z,+,\cdot)$, so $\left(\Z,+,X_{\sqrt{2}} (\frac{1}{3})\right)$ is definable in $(\Z,+,\cdot)$.
\end{proof}

\begin{example} \label{ex:int}
The group $\widetilde{\G}$ from Corollary \ref{cor:intt} is not equal to ${\G^*}^{00}_{\emptyset}$, i.e. it strictly contains ${\G^*}^{00}_{\emptyset}$. Indeed, by \ref{prop:divisible}, ${\G^*}^{00}_{\emptyset}$ is divisible, but $\widetilde{\G}$ is not (one can realize the type ``$x$ is divisible and $x=2y$, for some $y$ which is infinitesimally close to $\frac{\alpha}{2} \mod \alpha\Z$''). We can find a smaller type-definable subgroup $\H$ of ${\G^*}$ also containing ${\G^*}^{00}_{\emptyset}$. Actually define $\H = \bigcap_{m\in \N} m\widetilde{\G}$, alternatively \[\H = {\G^*}^0 \cap \bigcap_{n,m\in\N}m\cdot X_{\alpha}\left(\frac{t}{n}\right),\] where for a natural $m$ and a subset $S\subseteq G^*$ we denote by $m\cdot S$ the set $\{ms : s\in S\}$. Note that $\H$ has bounded index in ${\G^*}$. Indeed, since $m\cdot X_{\alpha}\left(\frac{t}{n}\right)$ is thick in $m\cdot {\G^*}$ which is of finite index in ${\G^*}$, the set $m\cdot X_{\alpha}\left(\frac{t}{n}\right)$ is thick in ${\G^*}$.

\begin{problem}
Is it the case that $\H = {\G^*}^{00}_{\emptyset}$?
\end{problem}

We will attempt to describe the compact (logic) topology on ${\G^*}^{0}/\H$. In order to do this, we change our structure to $\Rr = (\R,+,<,\Z,\alpha\Z)$. The structure $\G$ (from Corollary \ref{cor:intt}) is clearly definable in $\Rr$. Let $\Rr^*= (\R^*,+,<,\Z^*,\alpha\Z^*)$ be a sufficiently saturated extension of $\Rr$. Since the logic topology on ${\G^*}^{0}/\H$ is compact and Hausdorff, the topologies on ${\G^*}^{0}/\H$ determined from the structures $\Rr$ and $\G$ coincide (as a continuous bijection between compact Hausdorff spaces is a homeomorphism). It is immediate that inside $\Rr^*$ the group $\H$ has the following description \[\H = {\Z^*}^{0} \cap ({\alpha\Z^*}^0 + I),\] where $I$ is the subgroup of infinitesimals of $\Rr$.

Consider the following sequence of maps

\[\xymatrix{
{\G^*}^{0} = {\Z^*}^0 \ar@{^{(}->}[r]^-{\subseteq} & \Z^* \ar@{^{(}->}[r]^-{\subseteq} & \R^* \ar@{>>}[r] & \R^*/\left({\alpha\Z^*}^0+I\right)},\]
where the last function is the quotient map. Let $\varphi \colon {\G^*}^{0} \to \R^*/\left({\alpha\Z^*}^0+I\right)$ be the composite of these maps. The kernel $\ker(\varphi)$ is $\H$ by definition.

%\begin{equation*}
%\mathbb{Z}^{*0} \stackrel{\iota}{\longrightarrow} \mathbb{Z}^* \stackrel{\iota}{\longrightarrow} \mathbb{R}^* \stackrel{-/\alpha}{\longrightarrow} \mathbb{R}^* \stackrel{\pi}{\longrightarrow} \frac{\mathbb{R}^*}{I + \mathbb{Z}^{*0}}
%\end{equation*}

\begin{claim}
$\R^* = {\Z^*}^{0} + {\alpha\Z^*}^0 + I$. In other words $\varphi$ is surjective, so ${\G^*}^{0}/\H \cong \R^*/\left({\alpha\Z^*}^0+I\right)$.
\end{claim}
\begin{proof}
For $n\in\N$, $n\Z+\alpha n\Z$ is dense in $\R$. Hence, working in $\Rr$, the same is true in $\R^*$. Therefore ${\Z^*}^{0} + {\alpha\Z^*}^0 = \bigcap_{n\in\N}(n!\Z^*+\alpha n!\Z^*)$ is dense in $\R^*$ (by compactness, a countable decreasing intersection of definable dense sets is dense), so $\R^* = {\Z^*}^{0} + {\alpha\Z^*}^0 + I$.
%Since for each natural $k$, the set $k\cdot\alpha\Z$ is dense in $\R/k\cdot\Z$, the following clause is true in $\Rr$ \[\forall r\ \forall k\in\Z\ \exists x'\in\Z\ \exists y'\in\alpha\Z\  |r-kx'-ky'|<\frac{1}{k}. \tag{\textasteriskcentered}\]
%Let $r$ then be an arbitrary element of $\R^*$. Consider the following type $p(x,y)$ \[\left\{x\in\Z^*, y\in\alpha\Z^*,\left(\exists x'\in\Z^*\right)\left(\exists y'\in\alpha\Z^*\right)\left(x=kx'\wedge y=ky' \wedge |r-x-y|<\frac{1}{k}\right):k\in\N\right\}.\] Because of (\textasteriskcentered), the type $p$ is consistent, so the Claim follows.
\end{proof}

In fact ${\G^*}^{0}/\H$ is homeomorphic to $\R^*/\left({\alpha\Z^*}^0+I\right)$, and so to $\R^*/\left({\Z^*}^0+I\right)$ (all with the logic topology). Indeed, both are compact Hausdorff topological groups, and the induced mapping $\overline{\varphi}\colon {\G^*}^{0}/\H\to \R^*/\left({\alpha\Z^*}^0+I\right)$ is a continuous bijection. We will represent the logic topology on $\R^*/\left({\Z^*}^0+I\right)$ as the inverse (or projective) limit of the diagram $D$ of homomorphisms between copies of $S^1=\R/\Z$ which is described as follows:
\begin{enumerate}
\item for each $n \in \N$ there is a copy of $S^1$ labelled $S^{1(n)}$,
\item for each $a, b \in \N$ with $a$ dividing $b$ there is a map from $S^{1(b)}$ to $S^{1(a)}$ corresponding to multiplying by $b/a$.
\end{enumerate}

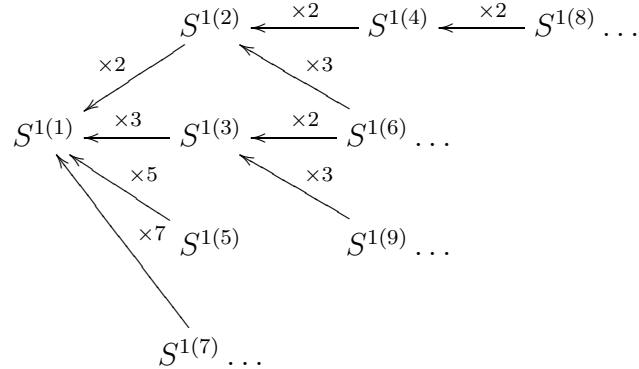
\[\xymatrix{
         & S^{1(2)} \ar[dl]_{\times 2} & S^{1(4)} \ar[l]_{\times 2} & S^{1(8)} \ldots \ar[l]_{\times 2}\\
S^{1(1)} & S^{1(3)} \ar[l]_{\times 3} & S^{1(6)} \ldots \ar[l]_{\times 2} \ar[ul]_{\times 3} \\
         & S^{1(5)} \ar[ul]_{\times 5} & S^{1(9)} \ldots \ar[ul]_{\times 3}\\
         & S^{1(7)} \ldots \ar[uul]_{\times 7}}\]

The limit $\varprojlim_D S^{1(n)}$ consists of points in the product topology respecting these maps. Each copy of $S^1$ can be thought of as a copy of ${\R^*}/({\Z^* + I})$; specifically $S^{1(n)}$ corresponds to ${\R^*}/({n\Z^* + I})$. Determining a point's position in each such quotient is equivalent to determining it in ${\R^*}/({{\Z^*}^0 + I})$, because $x - y \in {\Z^*}^0 + I$ if and only if $x - y \in n\Z^* + I$, for each $n\in\N$. This group is an abelian compact and torsion free topological group.
\end{example}

\begin{problem} \label{prob:qe}
What model theoretical properties does the structure \[\Rr = (\R,+,<,\Z,\alpha\Z)\] possess? Does it have a quantifier elimination in some reasonably natural language?
\end{problem}

We state two results related to the above problem. 
\begin{itemize}
\item By \cite{miler}, the structure $(\Q,+,<,0,1,\Z)$ admits a quantifier elimination in the language $(<,+,-,0,1,\lambda_q,\lfloor\rfloor)$, where $\lfloor\rfloor\colon\Q\to\Z$ is the floor function and $\lambda_q\colon\Q\to\Q$, $\lambda_q(x)=q\cdot x$. 
\item By a recent result of P. Hieronymi and M. Tychonievich, if $\alpha,\beta\in\R$ are such that $1$, $\alpha$ and $\beta$ are linearly independent over $\Q$, then $(\R,+,<,\N,\alpha\N,\beta\N)$ defines every open subset of $(0,1)^n$ for every $n \in \N$.
\end{itemize}

We now return to the more general setting.

Note that a group $G$ admits a homomorphism into $\R/\Z$ with dense image if and only if its abelianization $G^{\rm ab} = G/[G,G]$ has such a homomorphism. The \emph{exponent} of a group $G$ is the least common multiple (if it exists) of the orders of the elements of $G$. Otherwise we say that $G$ is of \emph{unbounded exponent}. Below we characterize the class of abelian groups admitting homomorphisms into $\R/\Z$ with dense image as the class of abelian groups with unbounded exponent, that is groups that contain elements of infinite order (\emph{non-torsion groups}), or that contain elements of arbitrary large finite order.

\begin{lemma} \label{lem:dense} Suppose $K$ is an abelian group which is either non-torsion or torsion of unbounded exponent. Then there is a homomorphism $g\colon K \to \R/\Z$ with dense image.
\end{lemma}
\begin{proof} Suppose that $K$ is a non-torsion group. First, we find a nontrivial homomorphism $g'\colon K \to \Q$. Consider $K \otimes \Q$. This is a nontrivial $\Q$-vector space. For a basis element $x$ of $K \otimes \Q$, let $\pi_x\colon K \otimes \Q \to \Q$ be the projection along $x$. There is a natural homomorphism $\otimes_1\colon K \to K \otimes \Q$ given by $a \mapsto a\otimes 1$. Let $g'$ be the composite of $\otimes_1$ and $\pi_x$. Then $g'\colon K \to \Q$ is nontrivial. Take any irrational $\alpha\in\R\setminus\Q$ and consider the mapping $g''\colon\Q\to \alpha\cdot\Q/\Z\subset\R/\Z$ given by $g''(q)=\alpha\cdot q + \Z$. The map $g''\circ g'$ has dense image in $\R/\Z$.

Suppose now that $K$ is a torsion group with unbounded exponent. We construct a homomorphism $g\colon K \to \Q/\Z$ with dense image. 

Write $\chi(a)$ for the order of $a$ in $K$.

We will construct this map explicitly. Well-order $K$ as $\{a_i: i \in \gamma\}$. Specifically choose this well-ordering such that for $i \in\N$ we have $\chi(a_i) > i \Pi_{j < i}{\chi(a_j)}$.

Define a nested chain of maps $\{g_i: i \in \gamma\}$. Each $g_i$ will be a homomorphism $\langle a_j: j < i \rangle \to \Q/\Z$. Call the domain of such a map $K_i$. We get these as follows.
\begin{itemize}
\item $g_0$ is the trivial map.
\item $g_\lambda = \bigcup_{\beta < \lambda}g_\beta$ when $\lambda$ is a limit.
\item $g_{\alpha + 1}$ is the extension of $g_\alpha$ to $K_{\alpha+1}$ given by taking $a_\alpha$ to $\frac{g_\alpha(ka_\alpha) + 1}{k}$, where $k$ is the least natural number such that $ka_\alpha \in K_\alpha$. 
\end{itemize}

We see that any $x \in K_{\alpha + 1} = \langle K_\alpha, a_\alpha \rangle = K_\alpha + \Z a_\alpha$ can be expressed in the form $m +s a_\alpha$, where $m \in K_\alpha$ and $s\in\Z$. We have $g_{\alpha + 1}(x) := g_\alpha(m) + s\frac{g_\alpha(ka_\alpha)+1}{k}$, which is a linear definition as $g_{\alpha + 1}(x + x') = g_{\alpha + 1}(x) + g_{\alpha + 1}(x')$.

Each map is well-defined. Suppose we have some $x \in K_{\alpha+1}$, where $x = m + sa_\alpha = m' + s'a_\alpha$. Then $g_{\alpha + 1}(m + sa_\alpha) - g_{\alpha + 1}(m' + s'a_\alpha) = g_\alpha(m) - g_\alpha(m') + (s - s')\frac{g_\alpha(ka_\alpha)+1}{k}$. We know that $k$ divides $s - s'$ because $(s - s')a_\alpha \in K_\alpha$, so the $\frac{(s-s')}{k}$ vanishes, since it is an integer, and we get $g_\alpha(m - m') + g_\alpha((s -s')a_\alpha)$, but $m - m' = (s' - s)a_\alpha$, so that sum vanishes.

This is sufficient to show that each map is a homomorphism, as is $g = \bigcup{g_\alpha}$.

To show that it is dense it suffices to show that for each $n$ there is some $q$ in the image with $q < 1/n$. Consider where $a_i$ goes. Since $\chi(a_i) > i \Pi_{j < i}{\chi(a_j)}$, and we know that for all $x\in K_\alpha$, $\chi(x) \leq \Pi_{j < i}{\chi(a_j)}$, so $k \geq i$. Element $a_i$ will then go to $\frac{r}{k}$ where $0 < r \leq \frac{3}{2}$, and this is at most $\frac{3}{2k}$.
\end{proof}

\begin{corollary} \label{cor:ab}
\begin{enumerate}
\item Suppose that an infinite finitely generated group $G$ has a finite index subgroup $H$ with infinite abelianization $H^{\rm ab}$. Then, the conclusion of Proposition \ref{prop:dense} holds in $G$. That is, there exists a thick subset $P$ of $G$ such that for a sufficiently saturated extension $\G^*$ of the structure $\G = (G,\cdot,P)$, ${\G^*}^{00}_{\emptyset}$ is strictly smaller that ${\G^*}^0_{\emptyset}$.
\item If $G$ is an infinite finitely generated virtually solvable group, then there exists a finite index subgroup $H<G$ with infinite $H^{\rm ab}$. Thus the conclusion of Proposition \ref{prop:dense} holds in $G$.
\end{enumerate}
\end{corollary}

\begin{proof}
$(1)$ By \cite[14.3.2]{km}, a finite index subgroup of a finitely generated group is also finitely generated. Hence $H$ and $H^{\rm ab}$ are finitely generated, so $H^{\rm ab}$ cannot be a torsion group with bounded exponent. Lemma \ref{lem:dense} gives us a homomorphism $H^{\rm ab}\to \R/\Z$ with dense image and hence a homomorphism $H \to \R/\Z$ with dense image. We can then apply Proposition \ref{prop:dense} to get thick $P=X(t)$, for some $t\in \left(0,\frac{1}{3}\right]$. Note that as $H=P^{n}$, for some $n\in\N$, $H$ is $\emptyset$-definable in $(G,\cdot,P)$, hence ${\G^*}^0_{\emptyset}={H^*}^0_{\emptyset} \gneqq {H^*}^{00}_{\emptyset}= {\G^*}^{00}_{\emptyset}$. Since $[G:H]$ is finite, $P$ is also thick in $G$.

$(2)$ Let $H_0<G$ be a solvable group of finite index. Consider the derived series $H_0^{(n)}$ of $H_0$. Since $G$ is infinite, there is a least $n$ such that $H_0^{(n+1)}=\left[H_0^{(n)},H_0^{(n)}\right]$ has infinite index in $H_0^{(n)}$. Take $H:=H_0^{(n)}$.
\end{proof}

Note that for a torsion group of bounded exponent (say $k$) there is no homomorphism to $\R/\Z$ with dense image, because no such homomorphism to $\R/\Z$ can take any value in $\left(0, \frac{1}{k}\right)$.

\begin{problem} \label{prob:cantor}
Describe the thick sets in infinite abelian groups of bounded exponent.
\end{problem}

%Suppose for example that we work in $G=\Z_2^{\omega}$, which is a compact topological group. If for each 
 
%We wish to investigate the thick sets in infinite abelian torsion groups of bounded exponent: in particular, whether for such a set $X$ it is true that $X - X$ must contain a subgroup of finite index.  The classical Steinhaus Theorem states that if $X\subseteq \Z_2^{\omega}$ has a positive measure, then $X-X$ contains some open subset. If $X\subseteq \Z_2^{\omega}$ is thick, then 
% Measurable thick sets have this property by a theorem of Steinhaus, and so in a model of set theory such as Solovay's model where all subsets of the Cantor space are measurable this is true for thick subsets of Cantor space.

% use Steinhaus thm: if X\subsetq 2^\omega has positive measure, then X-X contains an open subset, see "On Extensions of the Steinhaus Theorem for Distance Sets and Difference Sets", "A Generalization of Steinhaus' Theorem and Some Nonmeasurable Sets"

% every finite index subgroup of $2^\omega$ corresponds to an intersection of maximal ideals in the power set of $\omega, it may not be true, because if $H$ has index 2 in $2^\omega$, and $(1,1,\ldots)\in H$ then $H$ does not correspond to any ideal in $2^\omega$

\section{Optimality of the van der Waerden theorem}

The infinite version of the van der Waerden theorem \cite[2.1]{grs} states that for any partition of $\Z$ into finitely many pieces $A_0,\ldots,A_r$, one of the $A_i$ contains arbitrarily long arithmetic progressions. Then the set $A_i - A_i$ contains sets of the form $m[-N,N]\cap\Z$ for arbitrarily large $N$. Consider the set \[P = \bigcup_{0\leq i\leq r}A_i-A_i.\] Hence $P$ contains sets of the form $m[-N,N]\cap\Z$ for arbitrarily large $N$. Observe that this set is $(r+2)$-thick: consider a $r+2$-sequence: by the pigeonhole principle two elements of this sequence must fall in the same $A_i$, so their difference lies in $P$. The following question arises: does $P$ always contain $m\Z$ for some $m$? or maybe some finite sum $P+\ldots+P$ contains $m\Z$? Using the sets $X_{\alpha}(t)$ from the previous section we can prove that the answer is negative, so the van der Waerden theorem cannot be strengthened in this way.

\begin{theorem} \label{thm:waerden}
Suppose that $G$ is an abelian group which has a homomorphism $g\colon G\to \R/\Z$ with dense image as in Lemma \ref{lem:dense} and $n\in\N_{\geq 1}$.
\begin{enumerate}
\item There is a covering $G = A_0\cup\ldots\cup A_{3n}$ such that 
\[\ker(g) \text{ is the largest subgroup of }G\text{ contained in }P^n, \tag{\textreferencemark}\]
where $P = \bigcup_{i < 3n+1}A_i-A_i$, and $P^{\left\lfloor\frac{3n}{2} \right\rfloor+1} =G$. In particular, no finite index subgroup of $G$ is contained in $P^n$.
\item If moreover $G$ is non-torsion, then there is a covering $G = A_0\cup\ldots\cup A_{2n}$ with property {\rm (\textreferencemark)}, and such that $P^{n+1} = G$.
\end{enumerate}
\end{theorem} 
\begin{proof}
$(1)$ Denote $X(t)=g^{-1}(-t,t)$, for $t\in \left(0,\frac{1}{2}\right]$ (see the proof of Proposition \ref{prop:dense}). Let $Q=X\left(\frac{1}{6n}\right)$. Then, since $\left(-\frac{1}{6n},\frac{1}{6n}\right)$ is $(3n+1)$-generic in $\R/\Z\cong \left[-\frac{1}{2},\frac{1}{2}\right)$ and $\im(g)$ is dense in $\R/\Z$, the set $Q$ is $(3n+1)$-generic in $G$. Thus one can find $a_0,\ldots,a_{3n}$ from $G$ such that $A_i = Q+a_i$, for $i\leq 3n$ is a covering of $G$. Hence $P=Q+Q=X\left(\frac{1}{3n}\right)$, so $P^n = X\left(\frac{1}{3}\right)$ has the desired properties. Moreover, $P^{\left\lfloor\frac{3n}{2}\right\rfloor+1} = g^{-1}\left(\R/\Z\right)=G$.

$(2)$ If $G$ is non-torsion, then we may assume by modifying $g$ that the image of $g$ has trivial intersection with $\Q/\Z$ (see the first case of the proof of Lemma \ref{lem:dense}). It suffices to take $Q=X\left(\frac{2}{8n+1}\right)$ and $P=Q+Q$ (noting that $\left(-\frac{2}{8n+1},\frac{2}{8n+1}\right)$ is $(2n+1)$-generic in $\R/\Z$). Then $P^n = X\left(\frac{4n}{8n+1}\right)$, and cannot contain any dense subgroups of $\R/\Z$. Any $u$ with $g(u) \not = 0$ will have irrational image and generate a dense subset of $\R/\Z$, so $\ker(g)$ is the largest subgroup of $G$ contained in $P^n$.
\end{proof}

The Hales-Jewett \cite[2.2]{grs} theorem states that for any finite number of colours and finite length $l$ one can pick a number $n$ such that any coloured $n$-dimensional hypercube with sides of length $l$ contains a monochromatic combinatorial line, namely a set of $l$ points where some coordinates are fixed and the others vary together from $1$ to $l$.

\begin{corollary}
\begin{enumerate}
\item For a given $n$ there exists a covering $\Z = A_0\cup\ldots\cup A_{2n}$ such that $P^n$ has no nontrivial subgroups, where $P = \bigcup_{i < 2n+1}A_i-A_i$.
\item For a given $n$ there exists a covering $\Q^\omega = A_0\cup\ldots\cup A_{2n}$ such that $P^n$ has no nontrivial subgroups.

\item In general, for an abelian group $G$ with cardinality $\leq 2^{\aleph_0}$ there is a covering of $G$ by $2n+1$ sets such that $P^n$ contains no infinite cyclic subgroup of $G$.
\end{enumerate}
\end{corollary}
Part $(2)$ of the corollary is the closest analogue, in our context, to the objects appearing in the Hales-Jewett Theorem, and implies the same for $\Z^\omega$.
\begin{proof}
$(2)$ Apply Theorem \ref{thm:waerden}, noting that there is a map $\Q^\omega \rightarrow (\R\setminus \Q) \cup \{0\}$ with trivial kernel as the latter includes a continuum-dimensional $\Q$-vector space. 

$(3)$ If $G$ is torsion take $P=G$. Otherwise take $G\otimes \Q$, which is a vector space of a dimension at most $2^{\aleph_0}$ and apply $(2)$.
\end{proof}

\begin{question}
Each case of the preceding corollary gives us a covering by $2n + 1$ sets. Do the same results also hold for coverings by fewer sets?
\end{question}

\begin{proposition}
Let $\kappa$ be any cardinal greater than $2^{\aleph_0}$ and $G$ be an abelian torsion free group of cardinality $\kappa$. Then for any generic subset $P$ of $G$ (for example $P = \bigcup_{i < n}A_i-A_i$, for some partition $G = \bigcup_{i < n}A_i$) there is nontrivial $x\in G$ such that $x\Z$ is contained in $P - P$.
\end{proposition}
\begin{proof}
Let $S$ be a finite set such that $G=P+S$. For each $g$ in $G$, we define a sequence $(s(g)_k)_{k\in\N}$ of elements of $S$, chosen such that $k\cdot g - s(g)_k \in P$. The number of possible sequences is at most $2^{\aleph_0}$, so we can find distinct $g,h \in G$ such that $s(g)_k = s(h)_k$ for each $k$. Setting $x = h - g$ we have that for each $k$ \[kx = (k\cdot h - s(h)_k) - (k\cdot g - s(g)_k) \in P - P,\] and so $x\Z$ is contained in $P - P$.
\end{proof}

In fact, by a similar argument, if $|G|>2^{\lambda}$ one can find a $\lambda$-dimensional $\Z$-module in $P-P$.

\begin{acknowledgements}
We would like to thank both anonymous referees for their helpful feedback and especially for bringing Corollary \ref{cor:poi} to our attention. We would also like to thank Antongiulio Fornasiero for asking interesting questions related to the subject of the paper  and Anand Pillay for some ideas which we used in Example \ref{ex:int}.
\end{acknowledgements}

\end{document}